\documentclass{amsart}
\usepackage{amsfonts,amssymb,amscd,amsmath,enumerate,verbatim,calc}
\usepackage[all]{xy}

\newcommand{\wrt}{with respect to}
\newcommand{\aF}{\mathfrak{a} }
\newcommand{\bF}{\mathfrak{b} }

\newcommand{\n}{\mathfrak{n} }
\newcommand{\m}{\mathfrak{m} }

\newcommand{\cB}{\mathcal{B}}
\newcommand{\cO}{\mathcal{O}}
\newcommand{\cD}{\mathcal{D}}

\newcommand{\bP}{\mathbb{\partial}}

\newcommand{\rt}{\rightarrow}

\newcommand{\ov}{\overline}

\newcommand{\ba}{\mathbf{a}}

\newcommand{\image}{\operatorname{image}}

\newcommand{\hh}{\operatorname{ht}}

\newcommand{\coker}{\operatorname{coker}}

\newcommand{\Spec}{\operatorname{Spec}}
\newcommand{\mIso}{\operatorname{mIso}}

\newcommand{\soc}{\operatorname{soc}}
\newcommand{\Hom}{\operatorname{Hom}}

\newcommand{\Ass}{\operatorname{Ass}}
\newcommand{\Supp}{\operatorname{Supp}}

\theoremstyle{plain}

\newtheorem{thm}{Theorem}

\newtheorem{theorem}{Theorem}[section]
\newtheorem{corollary}[theorem]{Corollary}
\newtheorem{lemma}[theorem]{Lemma}
\newtheorem{proposition}[theorem]{Proposition}

\theoremstyle{definition}

\newtheorem{s}[theorem]{}
\newtheorem{remark}[theorem]{Remark}

\theoremstyle{remark}

\numberwithin{equation}{theorem}

\begin{document}

\title[De Rahm]{De Rahm cohomology of local cohomology modules}
 \author{Tony~J.~Puthenpurakal}
\date{\today}
\address{Department of Mathematics, Indian Institute of Technology Bombay, Powai, Mumbai 400 076}

\email{tputhen@math.iitb.ac.in}


\begin{abstract}
Let $K$ be a field of characteristic zero, $R = K[X_1,\ldots,X_n]$ and let $I$ be an ideal in $R$. Let $A_n(K)  = K<X_1,\ldots,X_n, \partial_1, \ldots, \partial_n>$  be the $n^{th}$ Weyl algebra over $K$. By a result due to Lyubeznik  the local cohomology modules $H^i_I(R)$ are holonomic $A_n(K)$-modules for each $i \geq 0$.
In this article we compute the De Rahm cohomology modules $H^*(\partial_1,\ldots,\partial_n ; H^*_I(R))$ for certain classes of ideals.
\end{abstract}

\maketitle

\section*{Introduction}
Let $K$ be a field of characteristic zero, $R = K[X_1,\ldots,X_n]$ and let $I$ be an ideal in $R$. For $i \geq 0$
let $H^i_I(R)$ be the $i^{th}$-local cohomology module of $R$ with respect to $I$. Let $A_n(K)  = K<X_1,\ldots,X_n, \partial_1, \ldots, \partial_n>$  be the $n^{th}$ Weyl algebra over $K$. By a result due to Lyubeznik, see \cite{Ly}, the local cohomology modules $H^i_I(R)$ are finitely generated $A_n(K)$-modules for each $i \geq 0$. In fact they
 are \textit{holonomic} $A_n(K)$ modules. In \cite{B} holonomic $A_n(K)$ modules are denoted as $\cB_n(K)$, the  \textit{Bernstein} class of left $A_n(K)$ modules.

Let $N$ be a left $A_n(K)$ module. Now $\bP = \partial_1,\ldots,\partial_n$ are pairwise commuting $K$-linear maps. So we can consider the De Rahm complex
$K(\bP;N)$. Notice that the De Rahm cohomology modules  $H^*(\bP;N)$ are in general only $K$-vector spaces. They are finite dimensional if $N$ is holonomic; see \cite[Chapter 1, Theorem 6.1]{B}.
 In particular $H^*(\bP;H^*_I(R))$ are finite dimensional $K$-vector spaces. In this paper we compute it for a few classes of ideals.

 Throughout let $K \subseteq L$ where $L$ is an algebraically closed field. Let $A^n(L)$ be the affine $n$-space over $L$. If $I$ is an ideal in $R$
 then
 \[
 V(I)_L = \{ \mathbf{a} \in A^n(L) \mid f(\mathbf{a}) = 0; \ \text{for all} \ f \in I \};
 \]
 denotes the variety of $I$ in $A^n(L)$. By Hilbert's Nullstellensatz $V(I)_L$ is always non-empty. We say that an ideal $I$ in $R$ is zero-dimensional  if $\ell (R/I)$ is finite and non-zero (here $\ell(-)$ denotes length). This is equivalent to saying that $V(I)_L$ is a finite non-empty set. If $S$ is a finite set then let $\sharp S$ denote the number of elements in $S$.
 Our first result is
 \begin{thm}
 Let $I \subset R$ be a zero-dimensional ideal. Then $H^i(\bP; H^n_I(R)) = 0$ for $i < n$ and
 \[
 \dim_K H^n(\bP ; H^n_I(R)) = \sharp V(I)_L
 \]
 \end{thm}

  For homogeneous ideals it is best to consider their vanishing set in a projective case. Throughout let $P^{n-1}(L)$ be the projective $n-1$ space over $L$. We assume $n \geq 2$. Let $I$ be a homogeneous ideal in $R$. Let
  \[
  V^*(I)_L = \{ \mathbf{a} \in P^{n-1}(L) \mid f(\mathbf{a}) = 0; \ \text{for all} \ f \in I \};
  \]
  denote the variety of $I$ in $P^{n-1}(L)$.   Note that $V^*(I)_L$ is a non-empty finite set if and only if $\hh(I) = n-1$. We prove
  \begin{thm}
  Let $I \subset R$ be a height $n-1$ homogeneous ideal. Then
 \begin{align*}
 \dim_K H^n(\bP ; H^{n-1}_{I}(R)) &= \sharp V^*(I)_L   - 1, \\
 \dim_K H^{n-1}(\bP ; H^{n-1}_{I}(R)) &= \sharp V^*(I)_L, \\
 H^i(\bP; H^{n-1}_{I}(R)) &= 0 \ \text{for} \ i \leq n-2.
 \end{align*}
  \end{thm}
  
Altough I am unable to find a reference it is known that if
 $M$ is holonomic then $H^i(\bP, M) = 0$ for $i < n - \dim M$; here $\dim M = $ dimension of support of $M$. However the known proof uses sophisticated techniques like derived categories. We give an elementary proof of it.
\begin{thm}
Let $M$ be a holonomic $A_n(K)$-module. Then  $H^i(\bP, M) = 0$ for $i < n - \dim M$.
\end{thm} 

The advantage of our proof is that it can also be easily generalized to prove analogous results for power series rings and rings of convergent power series rings over $\mathbb{C}$. To the best of my knowledge this is a new result.
\begin{thm}
Let $\cO_n$ be the ring $K[[X_1,\ldots,X_n]]$ or $\mathbb{C}\{\{x_1,\ldots,x_n\}\}$. Let $\cD_n = \cO_n[\partial_1,\ldots,\partial_n]$ be the ring of $K$-lineear differential operators on $\cO_n$. Let $M$ be a holonomic $\cD_n$-module. Then $H^i(\bP, M) = 0$ for $i < n - \dim M$.
\end{thm}

Let $M$ be a holonomic $A_n(K)$-module. By a result of Lyubeznik the set of associate primes of $M$ as a $R$-module is finite. Note that the set $\Ass_R(M)$ has a natural
partial order given by inclusion.
  We say $P$ is a \textit{maximal }isolated associate prime of $M$
 if  $P$ is a maximal ideal of $R$ and also a minimal prime of $M$. We set $\mIso_R(M)$ to be the set of all maximal isolated associate primes of $M$. We show

\begin{thm}
Let $M$ be a holonomic $A_n(K)$-module. Then
\[
\dim_K H^n(\bP;M) \geq  \sharp \mIso_R(M).
\]
\end{thm}
We give an application of Theorem 5. Let $I$ be an unmixed ideal of height $\leq n-2$. By Grothendieck vanishing theorem and the Hartshorne-Lichtenbaum vanishing theorem it follows that $H^{n-1}_I(R)$ is supported only at maximal ideals of $R$. By Theorem 5 we get
\[
\sharp \Ass_R H^{n-1}_I(R) \leq \dim_K H^n\left(\partial;H^{n-1}_I(R)\right).
\]

We now describe in brief the contents of the paper. In section 1 we discuss a few preliminary results that we need.
In section 2 we make a few computations. This is used in section 3 to prove Theorem 1. In section 4 we make some additional computations and use it in section  5 to prove Theorem 2. In section 6 we prove Theorem 5. In section 7 we prove Theorem 3. In section 8 we prove Theorem 4.
 \section{Preliminaries}
 In  this section  we discuss a few preliminary results that we need.
\begin{remark}
Altough all the results are stated for De-Rahm cohomology of a $A_n(K)$-module $M$, we will actually work with
De-Rahm homology. Note that $H_i(\bP, M) = H^{n-i}(\bP, M)$ for any $A_n(K)$-module. Let $S = K[\partial_1,\ldots,\partial_n]$. Consider it as a subring of $A_n(K)$. Then note that $H_i(\bP, M)$ is the $i^{th}$ Koszul homology module of $M$ with respect to $\bP$.
\end{remark} 
\s \label{mod-1}
Let $M$ be a holonomic $A_n(K)$-module. Then for $i = 0,1$ the De-Rahm  homology modules $H_i(\partial_n,M)$ are holonomic $A_{n-1}(K)$-modules, see \cite[1.6.2]{B}. 
 
The following result is well-known.
  \begin{lemma}\label{Koszul}
  Let $\bP = \partial_r, \partial_{r+1},\ldots,\partial_{n}$ and $\bP^\prime = \partial_{r+1},\ldots,\partial_{n}$. Let $M$ be a left $A_{n}(K)$-module. For each $i \geq 0$ there exist an exact sequence
  \[
  0 \rt H_0( \partial_r ; H_i(\bP^\prime;M)) \rt H_i(\bP; M) \rt H_1(\partial_r ; H_{i-1}(\bP^\prime ; M)) \rt 0.
  \]
  \end{lemma}

\s \label{change-variables} (linear change of variables). We consider a linear change of variables. Let $U_1,\ldots,U_n$ be new variables defined by
\[
U_i = d_{i1}X_1 + \cdots + d_{in}X_n + c_i \quad \text{for} \ i = 1,\ldots,n
\]
where  $d_{ij}, c_1,\ldots,c_n \in K$ are arbitrary and $D = [d_{ij}]$ is an invertible matrix.  We say that the change of variables is homogeneous if $c_i = 0$ for all $i$.

 Let $F = [f_{ij}] = (D^{-1})^{tr}$.  Using the chain rule it can be easily shown that
\[
\frac{\partial}{\partial U_i} = f_{i1}\frac{\partial}{\partial X_1} + \cdots + f_{in}\frac{\partial}{\partial X_n} \quad \text{for} \ i = 1, \ldots, n.
\]
In particular we have that for any $A_n(K)$ module $M$ an isomorphism of Koszul homologies
\[
H_i\left(\frac{\partial}{\partial U_1}, \cdots,\frac{\partial}{\partial U_n} ; M \right) \cong H_i\left(\frac{\partial}{\partial X_1}, \cdots,\frac{\partial}{\partial X_n} ; M \right)
\]
for all $i \geq 0$.

\s \label{Mayer}  Let $I,J$ be two ideals in $R$ with $J \supset I$ and let $M$ be a $R$-module. The inclusion $\Gamma_J(-) \subset \Gamma_I(-)$ induces, for each $i$, an $R$-module homomorphism
\[
\theta^i_{J,I}(M) \colon H^i_J(M) \rt H^i_I(M).
\]
If $L \supset J$ then we can easily see that
\begin{equation*}
\theta^i_{J,I}(M) \circ \theta^i_{L,J}(M)   = \theta^i_{L,I}(M). \tag{$\dagger$}
\end{equation*}
\begin{lemma}\label{theta}(with hypotheses as above)
If $M$ is a $A_n(K)$-module then
   the natural map $\theta^i_{J,I}(M)$ is $A_n(K)$-linear.
\end{lemma}
\begin{proof}
    Let $I = (a_1,\ldots,a_s)$. Using ($\dagger$) we may assume that
  $J = I + (b)$.
   Let $C(\ba;M)$ be the \v{C}ech-complex on $M$ \wrt \ $\ba$. Let $C(\ba,b;M)$ be the \v{C}ech-complex on $M$ \wrt \ $\ba,b$.
   Note that we have a natural short exact sequence of complexes of $R$-modules
   \[
  0 \rt C(\ba;M)_b[-1] \rt C(\ba,b;M) \rt  C(\ba;M) \rt 0.
   \]
Since $M$ is a $A_n(K)$-module it is easily seen that the above map is a map of complexes of $A_n(K)$-modules. It follows that the map $H^i(C(\ba,b;M)) \rt  H^i(C(\ba;M))$ is $A_n(K)$ linear. It is easy to see that this map is $\theta^i_{J,I}(M)$.
\end{proof}
\s\label{Mayer-2} Let $\aF,\bF$ be ideals in $R$ and let $M$ be an $A_n(K)$-module. Consider the Mayer-Vietoris sequence
is a sequence of $R$-modules
\[
\rt H^i_{\aF + \bF}(M) \xrightarrow{\rho^i_{\aF,\bF}(M)} H^i_{\aF}(M) \oplus H^i_{\bF}(M)\xrightarrow{\pi^i_{\aF,\bF}(M)} H^i_{\aF \cap \bF}(M) \xrightarrow{\delta^i}H^{i+1}_{\aF + \bF}(M) \rightarrow..
\]
Then for all $i \geq 0$ the maps $\rho^i_{\aF,\bF}(M)$ and  $\pi^i_{\aF,\bF}(M)$ are $A_n(K)$-linear.

To see this first note that since $M$ is a $A_n(K)$-module all the above local cohomology modules are $A_n(K)$-modules.
Further note that, (see \cite[15.1]{a7}),
\begin{align*}
\rho^i_{\aF,\bF}(M)(z) &= \left( \theta^i_{\aF+\bF,\aF}(z) , \theta^i_{\aF+\bF,\bF}(z) \right),\\
\pi^i_{\aF,\bF}(M)(x,y) &= \theta^i_{\aF,\aF \cap  \bF}(x) - \theta^i_{\bF,\aF \cap  \bF}(y).
\end{align*}
Using Lemma \ref{theta} it follows that $\rho^i_{\aF,\bF}(M)$ and   $\pi^i_{\aF,\bF}(M)$ are $A_n(K)$-linear maps.

\begin{remark}
Infact $\delta^i$ is also $A_n(K)$-linear for all $i \geq 0$; \cite{P}. However we will not use this fact
in this paper.
\end{remark}

\s\label{co-prime} Let $I_1,\ldots,I_n$ be proper ideals in $R$. Assume that they are pairwise co-maximal i.e., $I_i + I_j = R$ for $i \neq j$. Set $J = I_1\cdot I_2 \cdots I_n$. Then
for any $R$-module $M$ we have an isomorphism of $A_n(K)$-modules
\[
H^i_{J}(M) \cong \bigoplus_{j = 1}^{n}H^i_{I_j}(M) \quad \text{for all} \ i \geq 0.
\]
To prove this result note that $I_1$ and $I_2\cdots I_n$ are co-maximal. So it suffices to prove the result for $n = 2$. In this case we  use the Mayer-Vieotoris sequence of local cohomology, see \ref{Mayer-2}, to get
an isomorphism of $R$-modules
\[
\pi^i_{I_1,I_2}(R)\colon H^i_{I_1}(R) \oplus H^i_{I_2}(R)\rt H^i_{I_1 \cap I_2}(R).
\]
By \ref{Mayer-2} we also get that $\pi^i_{I_1,I_2}(R)$ is $A_n(K)$-linear.

\section{Some computations}
The goal of this section is to compute the Koszul homologies  $H_*(\partial_1,\ldots,\partial_n ; N)$
when $N= R$ and when $N = E$ the injective hull of $R/(X_1,\ldots,X_n) = K$.   It is well-known that
\[
E = \bigoplus_{r_1,\ldots,r_n \geq 0} K \frac{1}{X_1X_2\cdots X_n X_1^{r_1}X_2^{r_2}\cdots X_n^{r_n}}.
\]
Note that $E$ has the obvious structure as a $A_n(K)$-module with
\[
X_i \cdot\frac{1}{X_1\cdots X_n X_1^{r_1}\cdots X_n^{r_n}} = \begin{cases} \frac{1}{X_1\cdots X_n X_1^{r_1}\cdots X_i^{r_i-1}\cdots X_n^{r_n}}& \text{ if} \  r_i \geq 1,\\ 0& \text{otherwise.}       \end{cases}
\]
and
\[
\partial_i \cdot \frac{1}{X_1\cdots X_n X_1^{r_1}\cdots X_n^{r_n}} = \frac{-r_i -1}{X_1\cdots X_n X_1^{r_1}\cdots X_i^{r_i +1}\cdots X_n^{r_n}}
\]
It is convenient to introduce the following notation. For $i = 1, \cdots, n$ let $R_i = K[X_1,\ldots,X_i]$, $\m_i = (X_1,\ldots,X_i)$ and let $E_i $ be the injective hull of $R_i/\m_i = K$ as a $R_i$-module. Set $R_0 = E_0 = K$.
We prove
\begin{lemma}\label{E-basic-Lem}
$H_0(\partial_n; E_n) \cong E_{n-1}$ and $H_1(\partial_n;E_n) = 0$ as $A_{n-1}(K)$-modules.
\end{lemma}
\begin{proof}
Since $E_n$  is holonomic $A_n(K)$ module it follows that $H_i(\partial_n ; E_n)$ (for $i = 0,1$) are  holonomic $A_{n-1}(K)$-modules \cite[Chapter 1, Theorem 6.2]{B}. We first prove $H_1(\partial_n;E_n) = 0$. Let $t \in E_n$ with $\partial_n(t) =0$.
Let
$$ t = \sum_{ r_1,\ldots,r_n \geq 0} t_r \frac{1}{X_1\cdots X_n X_1^{r_1}\cdots X_n^{r_n}} \quad \text{with atmost finitely many $t_r$ non-zero.}  $$
Notice that
\[
\partial_n(t) = \sum_{ r_1,\ldots,r_n \geq 0} t_r \frac{-r_n -1}{X_1\cdots X_{n-1} X_n X_1^{r_1}\cdots X_{n-1}^{r_{n-1}}X_n^{r_n +1}}.
\]
Comparing coefficients we get that if $\partial_n(t) = 0$ then $t = 0$.

For computing $H_0(\partial_n; E_n)$ we first note that as $K$-vector spaces
\[
E_n = X \bigoplus Y;
\]
where
\begin{align*}
X &= \bigoplus_{r_1 ,\ldots,r_{n-1} \geq 0,r_n = 0} K \frac{1}{X_1X_2\cdots X_n X_1^{r_1} X_2^{r_2}\cdots X_{n-1}^{r_{n-1}}} \\
Y &= \bigoplus_{r_1,\ldots,r_{n-1} \geq 0,r_n \geq 1} K \frac{1}{X_1X_2\cdots X_n X_1^{r_1}X_2^{r_2}\cdots X_n^{r_n}}.
\end{align*}
For $r_n \geq 1$ note that
\[
\partial_n\left(\frac{1}{X_1X_2\cdots X_n X_1^{r_1}X_2^{r_2}\cdots X_n^{r_{n}-1}} \right) =  \frac{-r_n}{X_1X_2\cdots X_n X_1^{r_1}X_2^{r_2}\cdots X_n^{r_n}}.
\]
It follows that $E_n/\partial_n E_n = X$. Furthermore notice that $X \cong E_{n-1}$ as $A_{n-1}(K)$-modules. Thus we get
$H_0(\partial_n ; E_n) \cong E_{n-1}$.
\end{proof}
We now show that
\begin{lemma}\label{E-Lem}
For $c = 1,2,\ldots,n$ we have,
\[
H_i(\partial_c,\partial_{c+1},\cdots,\partial_n ; E_n) = \begin{cases} 0 & \text{for} \ i > 0 \\ E_{c-1}& \text{for} \ i = 0
\end{cases}
\]
\end{lemma}
\begin{proof}
We prove the result by induction on $t = n -c$. For $t = 0$ it is just the Lemma \ref{E-basic-Lem}. Let $t \geq 1$ and assume the result for $t -1$.
Let $\bP = \partial_c, \partial_{c+1},\ldots,\partial_{n}$ and $\bP^\prime = \partial_{c+1},\ldots,\partial_{n}$.  For each $i \geq 0$ there exist an exact sequence
  \[
  0 \rt H_0( \partial_c ; H_i(\bP^\prime; E_n)) \rt H_i(\bP; E_n) \rt H_1(\partial_c ; H_{i-1}(\bP^\prime ; E_n)) \rt 0.
  \]
By induction hypothesis $H_i(\bP^\prime ; E_n) = 0 $ for $i \geq 1$. Thus for $i \geq 2$ we have $H_i(\bP ; E_n) = 0$. Also note that by induction hypothesis
$H_0(\bP^\prime ; E_n) = E_c$. So we have
\[
H_1(\bP ; E_n) = H_1(\partial_c ; E_c) = 0 \quad \text{by Lemma \ref{E-basic-Lem}}.
\]
Finally again by Lemma \ref{E-basic-Lem} we have
\[
H_0(\bP ; E_n) = H_0(\partial_c ; E_c) = E_{c-1}.
\]
\end{proof}
As a corollary to the above result we have
\begin{theorem}\label{E-Th}
Let $\partial = \partial_1,\ldots,\partial_n$. Then
$H_i(\partial ; E_n) = 0$ for $i > 0$ and $H_0(\partial; E_n) = K$.\qed
\end{theorem}

We now compute the de Rahm homology $H_*(\partial; R)$. We first prove
\begin{lemma}\label{R-basic-Lem}
$H_0(\partial_n ; R_n) = 0$ and $H_1(\partial_n ; R_n) = R_{n-1}$
\end{lemma}
\begin{proof}
This is just calculus.
\end{proof}
The proof of the following result is similar to the proof of \ref{E-Lem}.
\begin{lemma}\label{R-Lem}
For $c = 1,2,\cdots,n$ we have,
\[
H_i(\partial_c,\partial_{c+1},\cdots,\partial_n ; R_n) = \begin{cases} 0 & \text{for} \ i = 0,1,\cdots,n-c \\ R_{c-1}& \text{for} \ i = n-c+1
\end{cases}
\]\qed
\end{lemma}

As a corollary to the above result we have
\begin{theorem}\label{R-Th}
Let $\partial = \partial_1,\ldots,\partial_n$. Then
$H_i(\partial ; R_n) = 0$ for $i < n$ and $H_n(\partial; R_n) = K$.\qed
\end{theorem}

We will need the following computation in part 2 of this paper.
\begin{lemma}\label{R-f}
Let $f$ be a non-constant squarefree polynomial in $R = K[X_1,\ldots,X_n]$. Let $\bP = \partial_1,\ldots,\partial_n$. Then
$H_n(\bP;R_f) = K$. Furthermore $H_n\left(\bP; H^1_{(f)}(R)\right) = 0$ and 
$$H_i\left(\bP; H^1_{(f)}(R)\right) \cong H_i(\bP;R_f) \quad \text{for} \ i < n.$$
\end{lemma}
\begin{proof}
Note that
\[
H_n(\bP;R_f) = \{ v \in R_f \mid \partial_i v = 0  \ \text{for all} \ i = 1,\ldots,n \}.
\]
Clearly if $v \in R_f$ is a constant then $\partial_i v = 0$ for all $i = 1,\ldots, n$.
By a linear change in variables we may assume that $f = X_n^s + \text{lower terms in } X_n$. Note that by \ref{change-variables} the de Rahm homology does not change.

Suppose if possible there exists a non-constant $v = a/f^r \in H_n(\bP;R_f)$ where $f$ does not divide $a$ if $r \geq 1$. Note that if $r = 0$ then $v \in H_n(\bP; R) = K$. So $v$ is a constant. So assume $r \geq 1$.
Since $\partial_n(v) = 0$ we get $f \partial_n(a) = ra\partial_n(f)$.

Since $f$ is squarefree we have $f = f_1\cdots f_m$ where $f_i$ are distinct irreducible polynomials. As
$f$ is monic in $X_n$ we have that $f_i$ is monic in $X_n$ for each $i$.

Since  $f \partial_n(a) = ra\partial_n(f)$ we have that $f_i$ divides $a\partial_n(f)$ for each $i$.
Note that  if $f_i$ divides $\partial_n(f)$ then $f_i$ divides $f_1\cdots f_{i-1} \partial_n(f_i)\cdot f_{i+1}\cdots f_m$.
Therefore $f_i$ divides $\partial_n(f_i)$ which is easily seen to be a contradiction since $f_i$ is monic in $X_n$. Thus $f_i$ divides $a$ for each $i = 1,\ldots,m$. Therefore $f$ divides $a$, which is a contradiction. Thus
$H_n(\bP;R_f)$ only consists of constants.

We have an exact sequence
\[
0 \rightarrow R \rightarrow R_f  \rightarrow H^1_I(R) \rightarrow 0.
\]
Notice $H_n(\bP,R) = H_n(\bP; R_f) = K$ and $H_{n-1}(\bP,R) = 0$ (see Theorem \ref{R-Th} and Lemma \ref{R-f}).
So we get $H_n(\bP,H^1_I(R)) = 0$. Also as $H_i(\bP,R) = 0$ for $i < n$ we get 
$$H_i\left(\bP; H^1_{(f)}(R)\right) \cong H_i(\bP;R_f) \quad \text{for} \ i < n.$$

\end{proof}
\section{Proof of Theorem 1}
In this section we prove Theorem 1.
Throughout $K \subseteq L$ where $L$ is an algebraically closed field. We first prove:

\begin{lemma}\label{basic}
Let $\m = (X_1-a_1,\cdots,X_n-a_n)$, where $a_1,\ldots, a_n \in K$, be a maximal ideal in $R = K[X_1,\ldots,X_n]$.
Let $\partial = \partial_1,\ldots,\partial_n$. Then
$H_i(\partial ; H^n_{\m}(R)) = 0$ for $i > 0$ and $H_0(\partial;H^n_{\m}(R) ) = K$.
\end{lemma}
\begin{proof}
Let $U_i = X_i - a_i$ for $i = 1,\ldots,n$. Then by \ref{change-variables}
\[
H_i\left(\frac{\partial}{\partial U_1}, \cdots,\frac{\partial}{\partial U_n} ; H^n_{\m}(R) \right) \cong H_i\left(\frac{\partial}{\partial X_1}, \cdots,\frac{\partial}{\partial X_n} ; H^n_{\m}(R) \right)
\]
for all $i \geq 0$.
Thus we may assume $a_1 = a_2 = \cdots = a_n = 0$. Finally note that $H^n_{\m}(R) = E$ the injective hull of $ R/\m = K$. So our result follows from Theorem \ref{E-Th}.
\end{proof}
We now give a proof of Theorem 1.

\begin{proof}[Proof of Theorem 1]
Notice
\begin{align*}
A_n(L) &= A_n(K)\otimes_K L \\
\text{and} \ S = L[X_1,\cdots,X_n] &= R\otimes_K L.
\end{align*}
So $A_n(L)$ and $S$ are faithfully flat extensions of $A_n(K)$ and $R$ respectively. It follows that
\[
H_i\left(\partial ; H^n_{IS}(S)\right) \cong H_i\left(\partial ; H^n_{I}(R)\right)\otimes_K L \quad \text{for all }\ i \geq 0.
\]
Thus we may as well assume that $K = L$ is algebraically closed.
Since $I$ is zero-dimensional we have
\[
\sqrt{I} = \m_1\cap \m_2 \cap \cdots \cap \m_r,
\]
where $\m_1,\ldots,\m_r$ are distinct maximal ideals and $r = \sharp V(I)_L$, the number of points in $V(I)_L$.
By \ref{co-prime} we have an isomorphism of $A_n(K)$-modules
\[
H^j_I(R) \cong \bigoplus_{i = 0}^{r} H^j_{\m_i}(R) \quad \text{for all } \ j \geq 0.
\]
In particular
we have that
\[
H_j\left(\partial; H^n_I(R) \right) = \bigoplus_{i = 0}^{r}H_j\left(\partial; H^n_{\m_i}(R)  \right).
\]
Since $K$ is algebraically closed each maximal ideal $\m$ in $R$ is of the form $(X_1-a_1,\ldots,X_n -a_n)$.
 The result follows from Lemma \ref{basic}.
\end{proof}

\section{some computations-II}
Let $R = K[X_1,\ldots,X_n]$ and let $P = (X_1,\ldots,X_{n-1})$. The goal of this section is to compute
$H_i(\partial; H^{n-1}_P(R))$ for all $i \geq 0$.

As before it is convenient to introduce the following notation. For $i = 1, \cdots, n$ let $R_i = K[X_1,\ldots,X_i]$, $\m_i = (X_1,\ldots,X_i)$ and let $E_i $ be the injective hull of $R_i/\m_i = K$ as a $R_i$-module.

Notice that $R_{n-1} \subseteq R_n $ is a faithfully flat extension. So
\[
R_n \otimes_{R_{n-1}} H^i_{\m_{n-1}}(R_{n-1}) \cong H^i_{\m_{n-1}R_n}(R_{n}) \quad \text{for all} \ i \geq 0.
\]
Thus
\[
H^{n-1}_{\m_{n-1}R_n}(R_{n}) = E_{n-1}[X_n ]  = \bigoplus_{j\geq 0}E_{n-1}X_n^j.
\]

We first prove the following:
\begin{lemma}\label{P-basic}
$H_1(\partial_n ; E_{n-1}[X_n]) = E_{n-1}$ and $H_0(\partial_n ; E_{n-1}[X_n]) = 0$.
\end{lemma}
\begin{proof}
Let $v \in E_{n-1}[X_n]_j$. So
\[
v = \frac{c}{X_1\cdots X_{n-1} X_1^{r_1}\cdots X_{n-1}^{r_{n-1}}}\cdot X_n^j
\]
for some $c \in K$ and $r_1,\ldots,r_{n-1} \geq 0$. Notice that
\[
\partial_n(v) = \begin{cases}
                \frac{cj}{X_1\cdots X_{n-1} X_1^{r_1}\cdots X_{n-1}^{r_{n-1}}}\cdot X_n^{j-1} & \text{if} \ j \geq 1, \\
                0 & \text{if} \ j = 0.
                    \end{cases}
\]
It follows that $H_1(\partial_n ; E_{n-1}[X_n]) = E_{n-1}$.

Let $v \in E_{n-1}[X_n]_j$ be a homogeneous element. So
\[
v = \frac{c}{X_1\cdots X_{n-1} X_1^{r_1}\cdots X_{n-1}^{r_{n-1}}}\cdot X_n^j
\]
for some $c \in K$ and $r_1,\ldots,r_{n-1} \geq 0$.
Let
\[
u = \frac{c}{j+1} \cdot \frac{1}{X_1\cdots X_{n-1} X_1^{r_1}\cdots X_{n-1}^{r_{n-1}}}\cdot X_n^{j+1}.
\]
Notice that $\partial_n(u) = v$.
Thus it follows that $H_0(\partial_n ; E_{n-1}[X_n]) = 0$.
\end{proof}
Next we prove
\begin{lemma}\label{P-Lem}
For $c = 1,2,\ldots,n$ we have,
\[
H_i(\partial_c,\partial_{c+1},\cdots,\partial_n ; E_{n-1}[X_n]) = \begin{cases} 0 & \text{for} \ i \neq 1 \\ E_{c-1}& \text{for} \ i = 1.
\end{cases}
\]
\end{lemma}
\begin{proof}
We prove the result by induction on $t = n -c$. For $t = 0$ it is just the Lemma \ref{P-basic}. Let $t \geq 1$ and assume the result for $t -1$.
Let $\bP = \partial_c, \partial_{c+1},\ldots,\partial_{n}$ and $\bP^\prime = \partial_{c+1},\ldots,\partial_{n}$.  For each $i \geq 0$ we have an exact sequence
  \[
  0 \rt H_0(\partial_c ; H_i(\bP^\prime; E_{n-1}[X_n])) \rt H_i(\bP; E_{n-1}[X_n]) \rt H_1(\partial_c ; H_{i-1}(\bP^\prime;E_{n-1}[X_n])) \rt 0.
  \]
So  $ H_i(\bP; E_{n-1}[X_n]) = 0$ for $i \geq 3$ and for $i = 0$. Notice that
\begin{align*}
H_2(\bP; E_{n-1}[X_n]) &= H_1(\partial_c; H_{1}(\bP^\prime ; E_{n-1}[X_n])) \\
                        &= H_1(\partial_c ; E_c); \ \text{(by induction hypothesis)}. \\
                        &= 0; \ \text{by Lemma \ref{E-basic-Lem}}.
\end{align*}
Similarly we have
 \begin{align*}
H_1(\bP; E_{n-1}[X_n]) &= H_0(\partial_c; H_{1}(\bP^\prime ; E_{n-1}[X_n])) \\
                        &= H_0(\partial_c ; E_c); \ \text{(by induction hypothesis)}. \\
                        &= E_{c-1}; \ \text{by Lemma \ref{E-basic-Lem}}.
\end{align*}
\end{proof}
As a corollary we obtain
\begin{theorem}\label{P-Th}
Let $R = K[X_1,\ldots,X_n]$ and let $P = (X_1,\ldots,X_{n-1})$. Let
$\bP = \partial_1,\ldots,\partial_{n}$. Then
\[
H_i(\bP ; H^{n-1}_{P}(R)) = \begin{cases} 0 & \text{for} \ i \neq 1 \\ K & \text{for} \ i = 1. \end{cases}
\]
\end{theorem}

\section{Proof of Theorem 2}
In this section we prove Theorem 2.
Throughout $K \subseteq L$ where $L$ is an algebraically closed field. We first prove:
\begin{lemma}\label{basic-2}
Let $Q = (X_1-a_1X_n, \cdots,X_{n-1}-a_{n-1}X_n )$, where $a_1,\ldots, a_{n-1} \in K$, be a homogeneous prime ideal in $R = K[X_1,\ldots,X_n]$.
Let $\partial = \partial_1,\ldots,\partial_n$. Then
$H_i(\bP ; H^{n-1}_{Q}(R)) = 0$ for $i \neq 1$ and $H_1(\bP ;H^{n-1}_{Q}(R) ) = K$.
\end{lemma}
\begin{proof}
Let $U_i = X_i - a_iX_n$ for $i = 1,\ldots,n-1$ and let $U_n = X_n$. Then by \ref{change-variables}
\[
H_i\left(\frac{\partial}{\partial U_1}, \cdots,\frac{\partial}{\partial U_n} ; H^n_{\m}(R) \right) \cong H_i\left(\frac{\partial}{\partial X_1}, \cdots,\frac{\partial}{\partial X_n} ; H^n_{\m}(R) \right)
\]
for all $i \geq 0$.
Thus we may assume $a_1 = a_2 = \cdots = a_{n-1} = 0$. The result follows from Theorem \ref{P-Th}.
\end{proof}

We now give 
\begin{proof}[Proof of Theorem 2]
As shown in the proof of Theorem 1 we may assume that $K = L$ is algebraically closed. We take $X_n = 0$ to be the hyperplane at infinity. After a homogeneous linear change of variables we may assume that there are no zero's of  $V(I)$ in the hyperplane $X_n = 0$; see \ref{change-variables}.
Thus
\[
\sqrt{I} = Q_1\cap Q_2\cap \cdots \cap Q_r
\]
where $r = \sharp V(I)$ and $Q_i = (X_1-a_{i1}X_n, \cdots, X_{n-1} - a_{i,n-1}X_n)$ for $i = 1,\ldots,r$.

We first note that $H^n_I(R) = 0$. This can be easily proved by induction on $r$ and using the Mayer-Vieotoris sequence.

We prove the result by induction on $r$. For $r = 1 $ the result follows from Lemma \ref{basic-2}. So assume $r \geq 2$ and that the result holds for $r - 1$. Set $J = Q_1\cap \cdots \cap Q_{r-1}$. Then $\sqrt{I} = J \cap Q_r$.
Notice that $\sqrt{Q_r + J} = \m = (X_1,\ldots,X_n)$. By  Mayer-Vieotoris sequence and the fact that $H^n_{Q_r}(R) =  H^n_J(R) = 0$ we get an exact sequence of $R$-modules
\[
0 \rightarrow H^{n-1}_J(R)\bigoplus H^{n-1}_{Q_r}(R) \xrightarrow{\alpha} H^{n-1}_{I}(R) \rightarrow H^n_\m(R) \rightarrow 0.
\]
By \ref{Mayer-2} $\alpha$ is $A_n(K)$ linear. Set $C = \coker \alpha$. So we have an exact sequence of $A_n(K)$-modules
\[
0 \rightarrow H^{n-1}_J(R)\bigoplus H^{n-1}_{Q_r}(R) \xrightarrow{\alpha} H^{n-1}_{I}(R) \rightarrow C \rightarrow 0.
\]
\textit{Claim:} $C \cong H^n_{\m}(R)$ as $A_n(K)$-modules.

First suppose the claim is true. Then note that
the result  follows from induction hypothesis and Lemma's \ref{basic}, \ref{basic-2}.

It remains to prove the claim. Note that $C \cong H^n_{\m}(R)$ as $R$-modules. In particular
$$\soc_R(C) = \Hom_R(R/\m , C) \cong \Hom_R(R/\m , H^n_{\m}(R) ) \cong K. $$
Let $e$ be a non-zero element of $\soc_R(C)$. Consider the map
\begin{align*}
\phi \colon A_n(K) &\rt C \\
               d &\mapsto de.
\end{align*}
Clearly $\phi$ is $A_n(K)$-linear. Since $\phi(A_n(K)\m) = 0$ we get an
$A_n(K)$-linear map
\[
\ov{\phi} \colon \frac{A_n(K)}{A_n(K)\m} \rt C.
\]
Note that $A_n(K)/A_n(K)\m \cong H^n_{\m}(R)$ as $A_n(K)$-modules.

To prove that $\ov{\phi}$ is an isomorphism, note that $\ov{\phi}$ is $R$-linear. Since $\ov{\phi}$ induces an isomprhism on socles we get that $\ov{\phi}$ is injective. As $H^n_{\m}(R)$ is an injective $R$-module and $\ov{\phi}$ is injective $R$-linear map we have that $C \cong \image{\ov{\phi}} \oplus \coker{\ov{\phi}}$ as $R$-modules.
Set $N = \coker{\ov{\phi}}$.
Note that $\soc_R(N) = 0$. Also note that as $R$-module $C$ is supported only at $\m$. So $N$ is supported only at $\m$.
Since $\soc_R(N) = 0$ we get that $N = 0$. So $\ov{\phi}$ is surjective. Thus $\ov{\phi}$ is an $A_n(K)$-linear isomorphism of $A_n(K)$-modules.
\end{proof}

\section{proof of Theorem 5}
In this section we prove Theorem 5.

\s Let $A$ be a Noetherian ring, $I$ an ideal in $A$ and let $M$ be an $A$-module, not necessarily finitely generated.
Set
\[
\Gamma_I(M) = \{ m \in M \mid I^sm = 0 \ \text{for some} \ s \geq 0 \}.
\]
 The following result is well-known. For lack of a suitable reference we give sketch of a proof here. When $M$ is finitely generated, for a proof of the following result see \cite[Proposition 3.13]{E}.
\begin{lemma}\label{mod-G}[with hyotheses as above]
\[
\Ass_A \frac{M}{\Gamma_I(M)} = \{ P \in \Ass_A M \mid P \nsupseteq I \}
\]
\end{lemma}
\begin{proof}\textit{(sketch)}
Note that if $P \in \Ass_A \Gamma_I(M)$ then $P \supseteq I$. It follows that if $P \in \Ass_A M$ and $P \nsupseteq I$
then $P \in \Ass_A M/\Gamma_I(M)$.

It can be easily verified that if  $P \in \Ass_A M/\Gamma_I(M)$ then $P \nsupseteq I$. Also note that if $P \nsupseteq I$ then $\Gamma_I(M)_P = 0$. Thus
\[
M_P \cong \left(\frac{M}{\Gamma_I(M)}\right)_P  \quad \text{if} \ P \nsupseteq I.
\]
The result follows.
\end{proof}
 We now give
\begin{proof}[Proof of Theorem 5]
First consider the case when $K$ is algebraically closed.
Set
\[
\Ass_A(M) = \mIso_R(M) \sqcup \left( \bigcup_{i =1}^{s} V(P_i) \cap \Ass_A(M)\right).
\]
Here $P_1,\ldots,P_s$ are minimal primes of $M$ which are not maximal ideals.

Set $I = P_1P_2\cdots P_s$. Note that $\Gamma_I(M)$ is a $A_n(K)$-submodule of $M$. Set $N = M/\Gamma_I(M)$. By Lemma
\ref{mod-G} we get that
\begin{align*}
\Ass_R N &= \{ P \in \Ass_R M \mid P \nsupseteq I \} \\
         &= \mIso(M).
\end{align*}
Let $\mIso(M) = \{ \m_1,\ldots,\m_r \}$. Set $J = \m_1\m_2\cdots\m_r$. Since $\m_1,\ldots,\m_r$ are comaximal we get by
\ref{co-prime} that as $A_n(K)$-modules
\[
\Gamma_J(N) = \Gamma_{\m_1}(N)\oplus \cdots \oplus \Gamma_{\m_r}(N).
\]
Set $E = N/\Gamma_J(N)$. By Lemma
\ref{mod-G} we get that $\Ass_R E = \emptyset$. So $E = 0$. Thus
\[
N = \Gamma_{\m_1}(N)\oplus \cdots \oplus \Gamma_{\m_r}(N).
\]
Note that
\[
 \Gamma_{\m_i}(N) = E_R(R/\m_i)^{s_i} = H^n_{\m_i}(R)^{s_i} \quad \text{for some} \ s_i \geq 1.
\]
Since $K$ is algebraiclly closed we have that for each $i =1,\ldots,r$ the maximal ideal $\m_i = (X_1 - a_{i1},\ldots, X_n - a_{in})$ for some $a_{ij} \in K$.
It follows from Lemma 3.1 that
\begin{align*}
H_i(\bP;N) &= 0 \ \text{for} \ i \geq 1 \\
\dim_K H_0(\bP;N) &= \sum_{i=1}^{r}s_i.
\end{align*}
The exact sequence $0 \rt \Gamma_I(M) \rt M \rt N \rt 0$ yields an exact sequence of
de Rahm homologies
\[
0 \rt H_0(\bP; \Gamma_I(M))  \rt H_0(\bP; M) \rt H_0(\bP; N) \rt 0;
\]
since $H_1(\bP; N) = 0$. The result follows. So we have proved the result when $K$ is algebraically closed.

Now consider the case when $K$ is \emph{not} algebraically closed. Let $L = \ov{K}$ the algebraic closure of $K$.
Note that $S = L[X_1,\ldots,X_n] = R \otimes_K L$ and $A_n(L) = A_n(K)\otimes_K L$. Further notice that
$M\otimes_K L$ is a holonomic $A_n(L)$-module. Also note that $M\otimes_R S = M\otimes_K L$.

\textit{Claim-1 :} $\sharp \mIso_S(M\otimes_R S) \geq \sharp \mIso_R(M)$.

We assume the claim for the moment. Note that $H_0(\bP, M)\otimes_K L = H_0(\bP, M\otimes_K L)$.
So
\[
\dim_K H_0(\bP, M) = \dim_L H_0(\bP, M\otimes_K L) \geq \sharp \mIso_S(M\otimes_R S) \geq \sharp \mIso_R(M).
\]
The result follows.

It remains to prove  Claim-1.
By Theorem 23.2(ii) of \cite{M} we have
\begin{equation*}
\Ass_S(M\otimes_R S) = \bigcup_{P \in \Ass_R(M)} \Ass_S \left(\frac{S}{PS} \right). \tag{$\dagger$}
\end{equation*}
Suppose $\m$ is an isolated maximal prime of $M$. Notice $S/\m S$ has finite length. It follows that
\[
\sqrt{\m S} = \m_1 \cap \m_2 \cap \cdots \cap \m_r;
\]
for some maximal ideals $\m_1, \m_2,  \cdots,  \m_r$ of $S$.

\textit{Claim-2 :} $\m_1, \m_2,  \cdots,  \m_r \in \mIso_S(M\otimes_R S)$.

Note that Claim-2 implies Claim-1. It remains to prove  Claim-2.

Suppose if possible some $\m_i \notin \mIso_S(M\otimes_R S)$. Then there exist $Q \varsubsetneqq \m_i$ and $Q \in \Ass_S(M\otimes_R S)$. Note that $Q$ is not a maximal ideal in $S$. By $(\dagger)$ we have that
\[
Q \in \Ass_S \left(\frac{S}{PS} \right) \quad \text{for some} \ P \in \Ass_R(M).
\]
Notice that as $Q$ is not a maximal ideal in $S$ we have that $P$ is not a maximal ideal in $R$.
Also note that by Theorem 23.2(i) of \cite{M} we have
\[
P = Q \cap R \subseteq \m_i \cap R = \m.
\]
Thus $\m$ is not an isolated maximal prime of $M$, a contradiction.
\end{proof}
An application of Theorem 5 is the following result:
\begin{corollary}
Let $I$ be an unmixed ideal of height $\leq n-2$ in $R$. Then
\[
\sharp \Ass_R H^{n-1}_I(R) \leq \dim_K H_0\left(\bP, H^{n-1}_I(R) \right).
\]
\end{corollary}
\begin{proof}
We first show that $M = H^{n-1}_I(R)$ is supported only at maximal ideals of $R$. As $M$ is $I$-torsion it follows that any $P \in \Supp(M)$ contains $I$.

 We first show that if $\hh P \leq n-2 $ then $P \notin \Supp(M)$. Note $M_P = H^{n-1}_{IR_P}(R_P) =0$ by Grothendieck vanishing theorem as $\dim R_P = \hh P \leq n-2$. So $P \notin \Supp(M)$.
 
 Next we prove that $\hh P = n-1 $ then $P \notin \Supp(M)$.  
 Let $\widehat{R_P}$ be the completion of $R_P$ with respect to its maximal ideal.  As $I$ is unmixed we have $\dim R_P/I_P >0$. So $I\widehat{R_P}$ is not $P\widehat{R_P}$-primary. Therefore
  $$M_P\otimes_{R_P} \widehat{R_P} = H^{n-1}_{I\widehat{R_P}}(\widehat{R_P}) = 0,$$ by Hartshorne-Lichtenbaum Vanishing theorem.  As $\widehat{R_P}$ is a faithfully flat $R_P$ algebra we have $M_P = 0$.

Thus $M$ is supported at only maximal ideals of $R$. It follows that $\Ass_A(M) = \mIso_R(M)$. The result
now follows from Theorem 5.
\end{proof}
\section{proof of Theorem 3}
In this section we give an elementary proof of Theorem 3. Set 

\noindent$R_{n-1} = K[X_1,\ldots,X_{n-1}]$.

We begin by the following result on vanishing (and non-vanishing) of de Rahm homology of a simple $A_n(K)$-module. If $M$ is a simple $A_n(K)$-module then it is well-known that $\Ass_R(M)$ consists of a singleton set.
\begin{theorem}\label{basic-simple}
Let $M$ be a simple $A_n(K)$-module and assume $\Ass_R(M) = \{ P\}$. Set  $Q = P \cap R_{n-1}$. Then
\begin{align*}
H_0(\partial_n;M) = 0 &\implies P = QR, \\
H_1(\partial_n;M) \neq 0 &\implies P = QR.
\end{align*}
\end{theorem}
To prove the above theorem we need a criterion for an ideal $I$ to be equal to $(I\cap R_{n-1})R$. This is provided by the following:
\begin{lemma}\label{ec}
Let $I$ be an ideal in $R$. Set  $J = I \cap R_{n-1}$. Then the following are equivalent:
\begin{enumerate}[\rm (1)]
\item
$\partial_n(I) \subseteq I$.
\item
$I = J R$.
\item
Let $\xi \in I$. Let $\xi = \sum_{j = 0}^{m}c_jX_n^j$ where $c_j \in R_{n-1}$ for  $j = 0,\ldots,m$. Then $c_j \in I$ 
for each $j$.
\end{enumerate}
\end{lemma}
\begin{proof}
We first prove $(1) \implies (3)$. Let  $\xi \in I$. Let $\xi = \sum_{j = 0}^{m}c_jX_n^j$ where $c_j \in R_{n-1}$ for  $j = 0,\ldots,m$. 
Notice $\partial_n^m(\xi) = m!c_m$. So $c_m \in I$. Thus $\xi - c_mX_n^m \in  I$. Iterating we obtain that $c_j \in I$ for all $j$.

Notice that $(3) \implies (1)$ is trivial. We now show $(3) \implies (2)$.  Let $\xi \in I$. Let $\xi = \sum_{j = 0}^{m}c_jX_n^j$ where $c_j \in R_{n-1}$ for  $j = 0,\ldots,m$. By hypothesis $c_j \in I$
for each $j$. Notice $c_j \in I \cap R_{n-1} = J$. Thus $I \subseteq JR$. The assertion $JR \subseteq I$ is trivial. So $I = JR$.

Finally we prove that $(2) \implies (3)$. If  $b \in J$ and $r \in R$ then notice that if $br = \sum_{j = 0}^{m}c_jX_n^j$ where $c_j \in R_{n-1}$ for  $j = 0,\ldots,m$ then each $c_j \in J$. As $I = JR$ each $\xi \in I$ is a finite sum $b_1r_1 + \cdots+ b_sr_s$ where $b_i \in J$ and $r_i \in R$. The assertion 
follows.
\end{proof}
The following corollary is useful.
\begin{corollary}\label{ec-P}
Let $P$ be a prime ideal in $R$ and let $I$ be an ideal in $R$ with $\sqrt{I} = P$. If $\partial_n(I) \subseteq I$ then $P = (P\cap R_{n-1})R$.
\end{corollary}
\begin{proof}
Set $Q = P \cap R_{n-1}$. Let $\xi \in P$. Let $\xi = \sum_{j = 0}^{m}c_jX_n^j$ where $c_j \in R_{n-1}$ for  $j = 0,\ldots,m$. Notice $\xi^s \in I$ for some $s \geq 1$. Also $\xi^s = c_m^{s}X_n^{sm} + .. $ lower terms in $X_n$. By Lemma \ref{ec} we get that $c_m^s \in I$. It follows that $c_m \in P$.
Thus $\xi - c_mX_n^m \in  P$. Iterating we obtain that $c_j \in P$ for all $j$.
So by Lemma \ref{ec} we get that $P = QR$.
\end{proof}
We now give
\begin{proof}[Proof of Theorem \ref{basic-simple}]
First suppose $H_0(\partial_n, M) = 0$. Let $a \in M$ with $P = (0 \colon a)$. Say $\partial_n b = a$. Set $I = (0\colon b)$. 

We first claim that $I \subseteq P$. Let $\xi \in I^2$. Notice $\partial_n \xi = \xi \partial_n + \partial_n(\xi)$. Also note that $\partial_n(\xi) \in I$. So we have that $\partial_n \xi b = \xi a + \partial_n(\xi)b$. Thus $\xi a = 0$. So $\xi \in P$. Thus $I^2 \subseteq P$. As $P$ is a prime ideal we get that $I \subseteq P$.

Next we claim that $\partial_n(I) \subseteq I$. Let $\xi \in I$. We have $\partial_n \xi b = \xi a + \partial_n(\xi)b$. So $\partial_n(\xi)b = 0$. Thus $ \partial_n(\xi) \in I$.

Since $M$ is simple we have that $M = A_n(K)a$.  So $b = d a$ for some $d \in A_n(K)$. It can be easily verified that
there exists $s \geq 1$ with $P^sd \subseteq A_n(K)P$.  It follows that $P^s \subseteq I$. Thus $\sqrt{I} = P$. The result
follows from \ref{ec-P}. 

Next suppose $H_1(\partial_n; M) \neq 0$. Say $a \in  \ker \partial_n $ is non-zero. Set $J = (0 \colon a)$. Let $\xi \in J$.  Notice $\partial_n \xi a = \xi \partial_na + \partial_n(\xi)a$. Thus $\partial_n(\xi) a = 0$. Thus $\partial_n(J) \subseteq J$.

By hypothesis $M$ is simple and $\Ass_R(M) = \{ P \}$. Now $\Gamma_P(M)$ is a non-zero $A_n(K)$-submodule of $M$. As $M$ is simple we have that 
$M = \Gamma_P(M)$.  Thus $P^s a = 0$ for some $s \geq 1$. Thus $P^s \subseteq J$. Also note that for any $R$-module $E$ the maximal elements in the set
$\{ (0 \colon e) \mid e \neq 0 \}$ are associate primes of $E$. Thus $J = (0 \colon a) \subseteq P$. Therefore $\sqrt{J} = P$.  The result
follows from \ref{ec-P}. 
\end{proof}

\begin{remark}\label{cl-R}
Let $P$ be a prime ideal in $R$. Set $Q = P\cap R_{n-1}$. Then it can be easily seen that
\[
\hh_R P  - 1 \leq \hh_{R_{n-1}} Q  \leq \hh_R P.
\]
Furthermore $\hh_{R_{n-1}} Q  = \hh_R P$ if and only if $P = QR$.
\end{remark}

\begin{remark}\label{mod}
Let $M$ be a holonomic $A_n(K)$-module. Assume $M$ is $I$-torsion. Set $J = I \cap R_{n-1}$. Then for $i=0,1$ the Koszul homology modules $H_i(\partial_n, M)$ are $J$-torsion holonomic $A_{n-1}(K)$-modules. For holonomicity see 1.2. Also note the sequence
\[
0 \rt H_1(\partial_n, M) \rt M \xrightarrow{\partial_n} M \rt H_0(\partial_n, M) \rt 0
\]
is an exact sequence of $A_{n-1}(K)$-modules. It follows that $H_i(\partial_n, M)$ are $J$-torsion for $i = 0,1$.
\end{remark}

\s \label{dim-ineq} Let $M$ be a $R$-module, not-necessarily finitely generated. By $\dim M $ we mean dimension of support of $M$. We set $\dim 0 = - \infty$.
It can be easily seen that the following are equivalent:
\begin{enumerate}
\item
$\dim M \leq n-i$.
\item
$M_P = 0$ for all primes $P$ with
 $\hh P < i$.
\end{enumerate}

\s \label{dim-simple} Let $M$ be a holonomic $A_n(K)$-module. Let $c = \ell_{A_n(K)}(M)$. So we have a composition series
\[
0 = V_0 \subsetneq V_1 \subsetneq V_2 \subsetneq \cdots \subsetneq V_c = M. 
\]
For $i =1,\ldots,c$, $C_i = V_i/V_{i-1}$ are simple holonomic $A_n(K)$-modules. Let $\Ass C_i = \{ P_i \}$. Set $d_i = \hh P_i$ and let $d = \min_{i}\{ d_i\}$. Then $$\dim M = n -d.$$
To see this let $d_j = d$. Set $P = P_j$. Then $(C_j)_P \neq 0$. So $(V_j)_P \neq 0$. So $M_P \neq 0$. Thus $\dim M \geq n -d$.   If $Q \in  \Spec(R)$ with $\hh Q < d$ then note that $P_i \nsubseteq Q$ for all $i$. Therefore $(C_i)_Q = 0$ for all $i$. It follows that $M_Q = 0$. Therefore $\dim M \leq n -d$ by \ref{dim-ineq}. Thus
$\dim M = n - d$. 

To prove Theorem 3 by induction we need the following:
\begin{lemma}\label{dim-lemm}
Let 
\[
0 = V_0 \subsetneq V_1 \subsetneq V_2 \subsetneq \cdots \subsetneq V_c = M. 
\]
be a composition series of a holonomic-module $M$. 
For $i =1,\ldots,c$ set $C_i = V_i/V_{i-1}$.
Then
\begin{enumerate}[\rm(1)]
\item
$\displaystyle{\dim H_0(\partial_n; M) \leq \max_{i}\{\dim H_0(\partial_n; C_i) \} \leq \dim M.}$
\item
$\displaystyle{\dim H_1(\partial_n; M) \leq \max_{i}\{\dim H_1(\partial_n; C_i) \} \leq \dim M-1.}$
\end{enumerate}
\end{lemma}
\begin{proof}
For $i = 1,\ldots,c$ we have an exact sequence
\begin{align*}
0 &\rt H_1(\partial_n;V_{i-1})  \rt H_1(\partial_n;V_{i}) \rt H_1(\partial_n; C_{i}) \\
   &\rt H_0(\partial_n;V_{i-1})  \rt H_0(\partial_n;V_{i}) \rt H_0(\partial_n; C_{i}) \rt 0.
\end{align*}
Let $\Ass C_i = \{ P_i \}$ and $d_i = \hh P_i$. Set $Q_i = P_i \cap R_{n-1}$.

(1) We prove the first inequality.
Suppose  if possible $H_0(\partial_n; C_i) = 0$ for all $i$. Then by the above exact sequence we get $H_0(\partial_n;V_{i}) = 0$
for all $i$. So $H_0(\partial_n,M) = 0$. Therefore the first inequality holds in this case.

Now suppose $H_0(\partial_n; C_i) \neq 0$ for some $i$. Set 
$$ \max_{i}\{\dim H_0(\partial_n; C_i) \} = n-1-c \quad \text{for some} \ c  \geq 0. $$
If $c = 0$ then we have nothing to prove.
Now suppose $c > 0$. Let $P$ be a prime in $R$ with $\hh P < c$. Then $H_0(\partial_n; C_i)_P = 0$ for all $i$.
By the above exact sequence we get $H_0(\partial_n;V_{i})_P = 0$
for all $i$. So $H_0(\partial_n,M)_P = 0$. Thus by \ref{dim-ineq} we get $\dim H_0(\partial_n,M) \leq n-1-c$.

We now prove that $\dim H_0(\partial_n, C_i) \leq \dim M$ for all $i$. Set $ N_i = H_0(\partial_n, C_i)$. We have nothing to prove if $N_i = 0$. So assume $N_i \neq 0$. By \ref{mod}, $N_i$ is $Q_i$-torsion. By \ref{cl-R} we have
$\hh Q_i \geq d_i - 1$. If $Q$ is a prime ideal in $R_{n-1}$ with $\hh Q < d_i -1$ then $Q \nsupseteq Q_{i}$. So $(N_i)_Q = 0$. By \ref{dim-ineq} 
$$\dim N_i  \leq n-1 -(d_i -1) = n - d_i \leq \dim M.$$
Here the last inequality follows from \ref{dim-simple}.

(2). The proof of the first inequality is same as that in $(1)$. Set $W_i = H_1(\partial_n, C_i)$. We prove 
$\dim W_i \leq \dim M - 1$  for all $i$.

If $\dim M = 0$ then note that $d_i = n$ for all $i$. So $P_i$ is a maximal ideal in $R$. It follows that $P_i \neq Q_i R$. So by Theorem \ref{basic-simple} we get $W_i =0$.

Now assume $\dim M  \geq 1$. If $W_i = 0$ then we have nothing to prove. So assume $W_i \neq 0$. Then by 
Theorem \ref{basic-simple} we have  $P_i = Q_i R$. So by \ref{cl-R} $\hh Q_i = \hh P_i = d_i$. By \ref{mod} $W_i$ is $Q_i$-torsion.
If $Q$ is a prime ideal in $R_{n-1}$ with $\hh Q < d_i $ then $Q \nsupseteq Q_{i}$. So $(W_i)_Q = 0$. By \ref{dim-ineq} 
$$\dim W_i  \leq n-1 -d_i    \leq \dim M -1.$$
Here the last inequality follows from \ref{dim-simple}.
\end{proof}

We now give 
\begin{proof}[Proof of Theorem 3]
We prove by induction on $n$ that $H_i(\bP, M) = 0$ for $i > \dim M$.    We first consider the case when $n = 1$. 
We have nothing to prove when $\dim M = 1$. If $\dim M = 0$ then $M$ is only supported at maximal ideals.
Let 
\[
0 = V_0 \subsetneq V_1 \subsetneq V_2 \subsetneq \cdots \subsetneq V_c = M. 
\]
be a composition series of  $M$. 
For $i =1,\ldots,c$ set $C_i = V_i/V_{i-1}$. Let $P_i = \Ass C_i$. Then $P_i$ is a maximal ideal of $R$. By \ref{basic-simple} we have $H_1(\partial_1,C_i) = 0$ for all $i$. So $H_1(\partial_1,M) = 0$.

Now assume $n \geq 2$. Let $\ov{M} = H_0(\partial_n,M)$ and $M_0 = H_1(\partial_n,M)$.  Set $\bP^\prime = \partial_1,\ldots,\partial_{n-1}$. Then we have an exact sequence
\[
\cdots \rt H_{j+1}(\partial^\prime; \ov{M}) \rt H_{j-1}(\partial^\prime; M_0) \rt H_j(\partial;M) \rt H_j(\partial^\prime;\ov{M}) \rt \cdots
\]
By Lemma \ref{dim-lemm} we have $\dim \ov{M} \leq \dim M$ and $\dim M_0 \leq \dim M -1$.
So for $j > \dim M$ we have, by induction hypothesis, $H_j(\partial^\prime;\ov{M}) = 0$ and $H_{j-1}(\partial^\prime; M_0) = 0$.
So $H_j(\partial;M) = 0$.
\end{proof}

\section{proof of Theorem 4}
In this section we prove Theorem 4. We only prove it in the case of $\cO_n = K[[X_1,\ldots,X_n]]$. The case of convergent power series rings is similar. The proof of Theorem 4 follows in the same pattern as in proof of Theorem 3. 
Only Lemma 7.2, 7.3, 7.8 and Remark 7.4 need an explanation. 

\begin{remark}
 Let $M$ be a holonomic $\cD_n$-module. Then $H_1(\partial_n; M)$ is a holonomic $\cD_{n-1}$-module; see \cite{v1}. However 
$H_0(\partial_n; M)$ need not be a holonomic $\cD_{n-1}$-module; see \cite{v2}. Nevertheless there exists a change of variables such that $H_i(\partial_n; M)$ are holonomic $\cD_{n-1}$-modules for $i= 0,1$; see \cite{v3}.

Iteratively it follows that there exists a change of variables such that $H_i(\partial^\prime; M)$ is finite dimensional $K$-vector spaces for $i \geq 0$. Note that $H_i(\partial; M) \cong H_i(\partial^\prime; M)$ for all $i \geq 0$ it follows that $H_i(\partial; M)$ are finite dimensional $K$-vector spaces.
\end{remark}

We first generalize Lemma \ref{ec}.
\begin{lemma}\label{ec-4}
Let $I$ be an ideal in $\cO_n$. Set  $J = I \cap \cO_{n-1}$. Then the following are equivalent:
\begin{enumerate}[\rm (1)]
\item
$\partial_n(I) \subseteq I$.
\item
$I = J \cO_{n}$.
\item
Let $\xi \in I$. Let $\xi = \sum_{j = 0}^{\infty}c_jX_n^j$ where $c_j \in \cO_{n-1}$ for  $j \geq 0$. Then $c_j \in I$ 
for each $j$.
\end{enumerate}
\end{lemma}
\begin{proof}
$(1) \implies (3):$ Let $\xi = \sum_{j = r}^{\infty}c_j X_n^j \in I$ with  $c_j \in \cO_{n-1}$ for  $j \geq r$.
Put $v_r =\xi$ and $c_j^{(r)} = c_j$ for $j\geq r$.
Put 
$$v_{r+1} = v_{r} - \frac{1}{(r+1)!}X_n^{r+1}\partial_n^{r+1}(v_r) = c_r X_n^r + \sum_{j\geq r+2}c_j^{(r+1)}X_n^{j}.$$
Here $c_j^{(r+1)} \in \cO_{n-1}$ for  $j \geq r+2$. By hypothesis $v_{r+1} \in I$. 

Now suppose $v_r, v_{r+1},\ldots,v_{r+s} \in I$ have been constructed where
\[
v_{r+s} = c_r X_n^r + \sum_{j\geq r+s+1}c_j^{(r+s)}X_n^{j}.
\]
Put
$$v_{r+s+1} = v_{r+s} - \frac{1}{(r+s+1)!}X_n^{r+s+1}\partial_n^{r+s+1}(v_{r+s}) = c_r X_n^r + \sum_{j\geq r+s+2}c_j^{(r+s+1)}X_n^{j}.$$
Here $c_j^{(r+s+1)} \in \cO_{n-1}$ for  $j \geq r+s+2$. By hypothesis $v_{r+s+1} \in I$. 

Since $v_{r+s} \in I$ we have that $c_rX_n^r \in I + \m^{r+s+1}$ for all $s \geq 1$. By Krull's intersection theorem
we have $\cap_{s \geq 1}(I + \m^{r+s+1}) = I$. So $c_rX_n^r \in I$. Therefore
\[
c_r = \frac{1}{r!}\partial_n^r(c_rX_n^r) \in I
\]

Now notice that $\xi - c_rX_n^r = \sum_{j = r+1}^{\infty}c_j X_n^j \in I$. Iteratively one can prove that 
$c_j \in I$ for all $i \geq r$.

The assertion $(3) \implies (1)$ is trivial. We now show $(3) \implies (2)$. Let $\xi = \sum_{j = r}^{\infty}c_j X_n^j \in I$ with  $c_j \in \cO_{n-1}$ for  $j \geq r$. Then by hypothesis $c_j \in I$ for  $j \geq r$. Set $S = \cO_{n-1}[X_n]$.
So $\xi_m = \sum_{j = r}^{m} c_j X_n^j \in JS$ for all $m \geq r$. Let $\widehat \ $ denote completion \wrt \ $X_n$-adic toplogy. Note $\xi = \lim_{m} \xi_m \in \widehat{JS} = J\widehat{S} = J\cO_n$. It follows that $I \subseteq J\cO_n$.
The assertion $J\cO_n \subseteq  I$ is trivial. So $I = J\cO_n$.

The proof of $(2) \implies (3)$ is similar to the analogus assertion in Lemma \ref{ec}.
\end{proof}
We now generalize Lemma \ref{ec-P}.
\begin{corollary}\label{ec-P-4}
Let $P$ be a prime ideal in $R$ and let $I$ be an ideal in $R$ with $\sqrt{I} = P$. If $\partial_n(I) \subseteq I$ then $P = (P\cap R_{n-1})R$.
\end{corollary}
\begin{proof}
Set $Q = P \cap \cO_{n-1}$. Let $\xi \in P$. Let $\xi = \sum_{j = r}^{\infty}c_jX_n^j$ where $c_j \in \cO_{n-1}$ for  $j \geq r$. Notice $\xi^s \in I$ for some $s \geq 1$. Also $\xi^s = c_r^{s}X_n^{sr} + .. $ higher terms in $X_n$. By Lemma \ref{ec-4} we get that $c_r^s \in I$. It follows that $c_r \in P$.
Thus $\xi - c_rX_n^r \in  P$. Iterating we obtain that $c_j \in P$ for all $j \geq r$.
So by Lemma \ref{ec-4} we get that $P = QR$.
\end{proof}

\begin{remark}\label{simple-th4}
Theorem 7.1 generalizes to the case of $\cD_n$ modules. The proof is the same.
\end{remark}

\begin{remark}\label{ht-contrac}
We now genralize Remark 7.4. Let $P$ be a prime ideal in $\cO_n$. Set $Q = P\cap \cO_{n-1}$. It is elementary that
$$\hh_{\cO_{n-1}}Q \leq \hh_{\cO_n} P \quad \text{with equality if and only if} \ P = Q\cO_{n}. $$
However the assertion $\hh Q \geq \hh P -1$ requires a proof. I thank J. K. Verma for providing this proof.
Note that $\hh Q = \hh Q\cO_{n}$. Set $A = \cO_{n-1}/Q$ and  $B = \cO_{n}/Q\cO_{n} = A[[X_n]]$. Set $\n = P/Q\cO_{n}$.
Let $S$ be the non-zero elements of $A$. Then $\n \cap S = \emptyset$. So $\hh \n = \hh \n S^{-1}B$. Let $L =$ quotient field of $A$. Then $S^{-1}B = L[[X_n]]$. It follows that $\hh \n \leq 1$. Therefore $\hh P - \hh Q \leq 1$. The result follows.
\end{remark}
For stating our generalization of Lemma 7.8 we need the following result:
\begin{proposition}\label{add}
 Let $0 \rt N \rt M \rt L \rt 0$ be a short exact sequence 
of holonomic $\cD_n$-modules. The following are equivalent:
\begin{enumerate}[\rm (1)]
\item
$H_i(\partial_n; M)$  are  holonomic $\cD_{n-1}$-module for $i = 0,1$.
\item
$H_i(\partial_n; N)$,  $H_i(\partial_n; M)$  are   holonomic $\cD_{n-1}$-modules for $i = 0,1$.
\end{enumerate} 
\end{proposition}
\begin{proof}
 Let $E$ be a holonomic $\cD_n$-module. Then $H_1(\partial_n; E)$ is a holonomic $\cD_{n-1}$-module; see \cite{v1}. 
 Note that we have an exact sequence of $\cD_{n-1}$-modules
 \[
 H_1(\partial;L) \rt H_0(\partial;N) \rt H_0(\partial;M) \rt H_0(\partial;L) \rt 0.
 \]
   $(2) \implies (1):$   By the above exact sequence $H_0(\partial;M)$ is a holonomic $\cD_{n-1}$-module.
  
   We now prove $(1) \implies (2)$. Note that $H_1(\partial;L)$ is holonomic $\cD_{n-1}$-module. By the above exact sequence $H_0(\partial;N)$ is a holonomic $\cD_{n-1}$-module. Furthermore $H_0(\partial;L)$ is a subquotient of $H_0(\partial;M)$ and so it is holonomic.
\end{proof}
The correct statement which generalizes Lemma 7.8 is the following:
\begin{lemma}\label{7.9-th4}
Let 
\[
0 = V_0 \subsetneq V_1 \subsetneq V_2 \subsetneq \cdots \subsetneq V_c = M. 
\]
be a composition series of a holonomic-module $M$. 
For $i =1,\ldots,c$ set $C_i = V_i/V_{i-1}$. Let $C = \bigoplus_{i = 1}^{c}C_i$. Suppose we have a change of variables with $H_i(\partial_n; C)$ holonomic $\cD_{n-1}$ module for $i = 0,1$. Then
\begin{enumerate}[\rm(1)]
\item
$H_i(\partial_n;C_j)$ are holonomic $\cD_{n-1}$ module for $i = 0,1$ and $j = 1,\ldots,c$.
\item
$H_i(\partial_n; M)$  are  holonomic $\cD_{n-1}$-module for $i = 0,1$.
\item
$\displaystyle{\dim H_0(\partial_n; M) \leq \max_{i}\{\dim H_0(\partial_n; C_i) \} \leq \dim M.}$
\item
$\displaystyle{\dim H_1(\partial_n; M) \leq \max_{i}\{\dim H_1(\partial_n; C_i) \} \leq \dim M-1.}$
\end{enumerate}
\end{lemma}
\begin{proof}
The assertions (1) and (2) follow from Proposition \ref{add}. The proof of assertions (3) and (4) is similar to that of (1) and (2) in Lemma 7.8.
\end{proof}
We now give
\begin{proof}[Proof of Theorem 4]
Let 
\[
0 = V_0 \subsetneq V_1 \subsetneq V_2 \subsetneq \cdots \subsetneq V_c = M. 
\]
be a composition series of a holonomic-module $M$. 
For $i =1,\ldots,c$ set $C_i = V_i/V_{i-1}$. Let $C = \bigoplus_{i = 1}^{c}C_i$. Choose a change of variables with $H_i(\partial_n; C)$ holonomic $\cD_{n-1}$ module for $i = 0,1$. Then by Lemma \ref{7.9-th4} we have that 
$H_i(\partial_n;C_j)$ are holonomic $\cD_{n-1}$ module for $i = 0,1$ and $j = 1,\ldots,c$. Furthermore
$H_i(\partial_n; M)$  are  holonomic $\cD_{n-1}$-module for $i = 0,1$.

After this choice of variables the proof of Theorem 4 is now identical to proof of Theorem 3.
\end{proof}

\end{document}